\documentclass[12pt]{amsart}

\newtheorem{THM}{Theorem}
\newtheorem{LMA}[THM]{Lemma}
\newtheorem{PROP}[THM]{Proposition}
\newtheorem{CORO}[THM]{Corollary}

\numberwithin{equation}{section}

\usepackage{amsmath}
\usepackage{amssymb}
\usepackage{euscript,url}
\usepackage{hyperref}

\newcommand{\showon}{\begin{align*}}
\newcommand{\showoff}{\end{align*}}

\newcommand{\none}{\varnothing}
\newcommand{\drop}{\smallsetminus}
\newcommand{\symdiff}{\bigtriangleup}

\newcommand{\goesto}{\rightarrow}
\newcommand{\one}{\boldsymbol{1}}
\newcommand{\zero}{\boldsymbol{0}}
\newcommand{\Sym}{\mathfrak{S}}

\newcommand{\la}{\langle}
\newcommand{\ra}{\rangle}
\newcommand{\ev}[1]{\la #1 \ra} 
\newcommand{\lb}{\lbrack}
\newcommand{\rb}{\rbrack}

\newcommand{\A}{\EuScript{A}} \renewcommand{\a}{\mathbf{a}}

 \newcommand{\CC}{\mathbb{C}}

\newcommand{\G}{\mathfrak{G}} 

\renewcommand{\L}{\EuScript{L}}
\newcommand{\M}{\EuScript{M}}

\renewcommand{\P}{\EuScript{P}}

 \newcommand{\RR}{\mathbb{R}}
 \renewcommand{\S}{\EuScript{S}}

\newcommand{\U}{\EuScript{U}}\renewcommand{\u}{\mathbf{u}}
\newcommand{\V}{\EuScript{V}}\renewcommand{\v}{\mathbf{v}}
\newcommand{\w}{\mathbf{w}}
 \newcommand{\x}{\mathbf{x}}

 \newcommand{\y}{\mathbf{y}}

 \newcommand{\ZZ}{\mathbb{Z}}

\begin{document}

\title[Highly symmetric matroids]{Highly symmetric matroids, the strong Rayleigh
property, and sums of squares}

\author{Wenbo Gao}
\thanks{Research supported by an NSERC Undergraduate Research Award.}
\address{Department of Combinatorics and Optimization,\
University of Waterloo,\
Waterloo, Ontario, Canada\ \ N2L 3G1}
\email{\texttt{w6gao@math.uwaterloo.ca}}

\author{David G. Wagner}
\thanks{Research supported by NSERC Discovery Grant OGP0105392.}
\address{Department of Combinatorics and Optimization,\
University of Waterloo,\
Waterloo, Ontario, Canada\ \ N2L 3G1}
\email{\texttt{dgwagner@math.uwaterloo.ca}}


\begin{abstract}
We investigate the strong Rayleigh property of matroids for which the
basis enumerating polynomial is invariant under a Young subgroup of the symmetric
group on the ground set.  In general, the Grace-Walsh-Szeg\H{o} theorem can be
used to simplify the problem.  When the Young subgroup has only two orbits, such
a matroid is strongly Rayleigh if and only if an associated univariate polynomial has only
real roots.  When this polynomial is quadratic we get an explicit structural
criterion for the strong Rayleigh property.  Finally, if one of the orbits has
rank two then the matroid is strongly Rayleigh if and only if the Rayleigh difference
of any two points on this line is in fact a sum of squares.
\end{abstract}

\maketitle

\begin{center}
\textbf{DRAFT IN PROGRESS}
\end{center}

\section{Introduction.}

The strong Rayleigh property of a matroid is a real semi-algebraic condition which is motivated by its
connection with abstractions of physical properties of electrical networks
\cite{Br1,Br2,BW,COSW,CW,KPV,W,WW}.
Br\"and\'en \cite{Br1} shows that this condition is equivalent to ``stability''
or the ``half-plane property'' studied in in the references above, from which we
take the following facts.  Binary matroids are strongly Rayleigh if and only if
they are regular.  $\mathrm{GF}(3)$- or $\mathrm{GF}(4)$-representable matroids are strongly Rayleigh
if and only if they are sixth-root of unity matroids.  Uniform matroids and the V\'amos matroid are strongly
Rayleigh, and a few more examples are known.   Minors, duality, free extensions, and two-sums
preserve the property.
Determining whether or not a given matroid is strongly Rayleigh is often a
challenging problem.  Here we provide infinitely many new examples of strongly
Rayleigh matroids, although they all have a very simple structure. 

Given a matroid $\M$ on the ground set $E$, let $\y=\{y_h:\ h\in E\}$ be algebraically independent commuting
indeterminates, and for $S\subseteq E$ let $\y^S = \prod_{h\in S} y_h$.  The \emph{basis enumerator} of $\M$
is the polynomial $M(\y) = \sum_{B} \y^B$ with the sum over the set of all bases of $\M$.  For distinct elements
$e,f\in E$, one can write
$$
M = M(\y) = M^{ef} + y_e M_e^f + y_f M_f^e + y_e y_f M_{ef}
$$
uniquely, in which the polynomials $M^{ef}, M_e^f, M_f^e, M_{ef}$ do not involve the variables $y_e$ or $y_f$.
The \emph{Rayleigh difference of $e$ and $f$ in $\M$} is
$$
\Delta M\{e,f\} = M_e^f M_f^e - M_{ef} M^{ef}.
$$
The matroid $\M$ has the \emph{strong Rayleigh property} provided that for every pair of distinct elements
$e,f\in E$, then $\Delta M\{e,f\}(\a)\geq 0$ for all $\a\in\RR^{E\drop\{e,f\}}$.
This definition extends naturally to any multiaffine polynomial
$Z(\y)=\sum_{S\subseteq E} \varphi(S) \y^S$ with real coefficients.

Let $\pi$ be a partition of the set $E$, and let $\Sym_\pi$ be the \emph{Young
subgroup} of all permutations of $E$ which leave each block of $\pi$ invariant
as a set.  We consider the class of matroids which are invariant under some
Young subgroup of the ground set.  The restriction of such a matroid to any
orbit is thus a uniform matroid.  As seen in Proposition \ref{GWS}, the
Grace-Walsh-Szeg\H{o} theorem allows us to reduce the number of variables in the
basis enumerator from $|E|$ to $|\pi|$, substantially simplifying the problem.
When $\pi$ has only two blocks, we obtain the following.

\begin{THM} \label{TWOFLATS}
Let $\M$ be the matroid of rank $r$ with point set $E=S\cup T$ partitioned
into disjoint flats $S$ of rank $s$ and $T$ of rank $t$, and with $|S|=a$ and
$|T|=b$, with points in the most general position possible.
Then $\M$ is strongly Rayleigh if and only if the polynomial
$$
P_\M(x)=\sum_{i=r-t}^{s} \binom{a}{i} \binom{b}{r-i} x^{i-r+t}
$$
has only real (nonpositive) roots.
\end{THM}

In Theorem \ref{TWOFLATS}, if $s+t=r$ then $\M$ is a direct sum, and if
$s+t=r+1$ then $\M$ is a two-sum, of uniform matroids.
These cases were already known to be strongly Rayleigh.  When $s+t\geq r+2$ the
condition in Theorem \ref{TWOFLATS} is nontrivial, and for $s+t=r+2$ we get
the following explicit criterion from the discriminant of a quadratic
univariate polynomial.

\begin{CORO} \label{CORO2FLATS}
Adopt the notation of Theorem \ref{TWOFLATS}, and assume that $s+t=r+2$.
Then $\M$ is strongly Rayleigh if and only if
$$
\frac{(a-s+2)(b-t+2)}{(a-s+1)(b-t+1)} \geq \frac{4(s-1)(t-1)}{st}.
$$
\end{CORO}

Finally, we consider whether a Rayleigh difference in $\M$ is in fact a sum of
squares of polynomials.

\begin{THM}\label{MAIN}
Let $r\geq 3$, $\ell\geq 1$, and $a\geq r-2$ be integers.
Let $\M=\M(r,\ell,a)$ be the rank
$r$ simple matroid with $\ell+2$ points on a line $L\cup\{e,f\}$ and a set $A$ of
$a$ points in general position relative to this line.  Then $\M(r,\ell,a)$ is
strongly Rayleigh if and only if either
$(r-2)\ell\leq 2$ or
$$
a \leq \A(r,\ell) = r + \frac{2(r+\ell+1)}{(r-2)\ell-2}.
$$
Moreover, the Rayleigh difference $\Delta M\{e,f\}$ is a sum of squares
if and only if this condition holds.
\end{THM}

Table \ref{A-TABLE} indicates the upper bound of Theorem \ref{MAIN} for
small values of $r$ and $\ell$.
\begin{table}
$$
\begin{array}{|r|rrrrrrrrrrrr|}
\hline
r \diagdown \ell
   & 1      & 2      & 3  & 4  & 5 & 6 & 7 & 8 & 9 & 10 & 11 & 12 \\
\hline
3  & \infty & \infty & 17  & 11  & 9  & 8  & 7  & 7  & 6  & 6  & 6  & 6  \\
4  & \infty & 11     & 8   & 7   & 6  & 6  & 6  & 5  & 5  & 5  & 5  & 5  \\
5  & 19     & 9      & 7   & 7   & 6  & 6  & 6  & 6  & 6  & 6  & 6  & 6  \\
6  & 14     & 9      & 8   & 7   & 7  & 7  & 7  & 7  & 6  & 6  & 6  & 6  \\
7  & 13     & 9      & 8   & 8   & 8  & 8  & 7  & 7  & 7  & 7  & 7  & 7  \\
8  & 13     & 10     & 9   & 9   & 9  & 8  & 8  & 8  & 8  & 8  & 8  & 8  \\
9  & 13     & 11     & 10  & 10  & 9  & 9  & 9  & 9  & 9  & 9  & 9  & 9  \\
10 & 14     & 11     & 11  & 11  & 10 & 10 & 10 & 10 & 10 & 10 & 10 & 10 \\
11 & 14     & 12     & 12  & 11  & 11 & 11 & 11 & 11 & 11 & 11 & 11 & 11 \\
12 & 15     & 13     & 13  & 12  & 12 & 12 & 12 & 12 & 12 & 12 & 12 & 12 \\
\hline
\end{array}
$$
\caption{The upper bound $a \leq \A(r,\ell)$ from Theorem \ref{MAIN}.}  \label{A-TABLE}
\end{table}

We asssume familiarity with matroid theory \cite{Ox} and the rudiments of symmetric
functions \cite{St}.  In Section 2 we discuss some preliminary material,
prove Theorem \ref{TWOFLATS}, and apply our method in the case of uniform matroids for
later use.  In Section 3 we apply the method to prove Theorem \ref{MAIN}.
In Section 4 we discuss some further potential applications which
might be tractable.

We thank Petter Br\"and\'en, Chris Godsil, Mario Kummer, Levent Tun\c{c}el,
and Cynthia Vinzant for interesting and helpful conversations and correspondence.

\renewcommand{\c}{\mathbf{c}}

\section{Preliminaries.}

\subsection{The strong Rayleigh property}

A \emph{form} is a homogeneous polynomial; a $d$-form has degree $d$; a
polynomial is \emph{multiaffine} if each variable occurs to at most the
first power.  A polynomial $F(\y)\in\RR[\y]$ is \emph{positive
semidefinite} (PSD) when $F(\a)\geq 0$ for all $\a\in\RR^E$.
A matroid is strongly Rayleigh exactly when every Rayleigh difference is a PSD form.
A \emph{sum-of-squares} (SOS) polynomial $F(\y)$
is one for which there are polynomials $p_i(\y)\in\RR[\y]$ for $1\leq i\leq n$ such that

$F(\y)=\sum_{i=1}^n p_i(\y)^2$.  It is not hard to see that a SOS $2d$-form
is a sum of squares of $d$-forms.  Clearly SOS polynomials are PSD, but the converse is
false.  The relationship between these concepts is the source of Hilbert's 17th
problem, and is a subject of continuing interest \cite{Sch}.

In the following proposition, a polynomial $Z(y_1,...,y_m)$ is \emph{stable}
provided that for any $w_i\in\CC$ with $\mathrm{Im}(w_i)>0$ for all $1\leq i\leq
m$, then $Z(w_1,...,w_m)\neq 0$.  Note that a univariate real polynomial is
stable if and only if it has only real roots.
\begin{PROP} \label{BWW}
Let $Z(y_1,...,y_m)$ be a multiaffine polynomial with real coefficients, and let
$E=\{1,...,m\}$.  The following conditions are equivalent:\\
\textup{(a)}\ $Z$ is stable.\\
\textup{(b)}\ $Z$ has the strong Rayleigh property.\\
\textup{(c)}\ for every index $g\in E$, both $Z_g$ and $Z^g$
have the strong Rayleigh property, and either $m\leq 1$ or for some pair of indices $\{e,f\}
\subseteq E$, $\Delta Z\{e,f\}(\a) \geq 0$ whenever $\a\in\RR^{E\drop\{e,f\}}$.\\
\textup{(d)} Either $m\leq 1$ or there exists a pair of indices $\{e,f\}$ such
that $Z_e$, $Z^e$, $Z_f$, $Z^f$ are all strongly Rayleigh, and
$\Delta Z\{e,f\}(\a) \geq 0$ whenever $\a\in\RR^{E\drop\{e,f\}}$.
\end{PROP}
\begin{proof}
The equivalence of (a) and (b) is Theorem 5.6 of Br\"and\'en \cite{Br1}.  That (b) is
equivalent to (c) is proved in Theorem 3 of \cite{WW}.  Condition (c) clearly implies (d).
That (d) implies (b) is part of Theorem 3.1 of \cite{Wag}, but the argument there
is sketchy.  Here we show that (d) implies (c), bridging this gap.
First, it follows from (\ref{BWW-PROOF}) below that if $Z$ is a strongly Rayleigh
multiaffine polynomial, then every deletion $Z^g$ and every contraction $Z_g$ is
also strongly Rayleigh.

We prove that (d) implies (c) by induction on $m$, with the basis of induction $m\leq 2$ being
trivial.  For the induction step assume that (d) holds, let $m\geq 3$, let $\{e,f\}$ be a pair of
indices as in (d), and let $g\not\in\{e,f\}$ be any third index.  To prove
(c) we need only show that $Z_g$ and $Z^g$ are strongly Rayleigh.

Expanding
$\Delta Z\{e,f\}$ as a quadratic in $y_g$, we have
\begin{equation} \label{BWW-PROOF}
\Delta Z\{e,f\} = y_g^2 \Delta Z_g\{e,f\} + y_g Q + \Delta Z^g\{e,f\} 
\end{equation}
for some polynomial $Q$ not involving $y_e$, $y_f$, or $y_g$.
Setting $y_g=0$ in (\ref{BWW-PROOF}), condition (d) implies that $\Delta
Z^g\{e,f\}(\a)\geq 0$ whenever $\a\in\RR^{E\drop\{e,f,g\}}$.
Condition (d) also implies that all of $Z_e^g$, $Z^{eg}$, $Z_f^g$, $Z^{fg}$ are strongly Rayleigh.
So $Z^g$ satisfies (d), and so by induction on $m$,  condition (c) holds for $Z^g$,
so that $Z^g$ is strongly Rayleigh since (c) implies (b).
By considering the limit of $y_g^{-2}\Delta Z\{e,f\}$ as $y_g\goesto\infty$, a
similar argument shows that $Z_g$ is strongly Rayleigh.  Thus, condition (c)
holds.

This completes the induction step, and the proof.
\end{proof}

If every Rayleigh difference of $\M$ is a square of a polynomial, then $\M$ is
certainly strongly Rayleigh.  (Regular matroids have this property.)
By Theorem 5.5 of \cite{KPV}, the basis enumerator $M(\y)$ of such a matroid has
a ``definite determinantal representation''.  The V\'amos matroid $\V_8$ is known
to be strongly Rayleigh \cite{WW}, but its basis enumerator does not have a definite
determinantal representation \cite{Br2}.  In fact, for some pair of elements of
$\V_8$ the Rayleigh difference is a PSD form but not a SOS form \cite{K1,KPV}.

\subsection{Highly symmetric matroids}

The following follows from the Grace-Walsh-Szeg\H{o} theorem; see \cite{BW,Wag}.
\begin{PROP} \label{GWS}
Let $Z(\y)\in\RR[\y]$ be a polynomial and $\pi$ a partition of the set $E$.
Assume that $Z$ is invariant under every permutation in $\Sym_\pi$.
Let $\x=\{x_B: B\in\pi\}$ be indeterminates indexed by the blocks of $\pi$, and
define $\beta_\pi:\RR[\y]\goesto\RR[\x]$ by
setting $\beta_\pi(y_h)=x_B$ for all $B\in\pi$ and $h\in B$, and algebraic
extension.  Then $Z(\y)$ is stable if and only if $\beta_\pi Z(\x)$ is stable.
\end{PROP}

For an integer partition $\lambda$ and subset $S \subseteq E$, let $e_\lambda(S)$
be the elementary symmetric function of shape $\lambda$ in the variables
$\{y_h:\ h\in S\}$, and similarly for monomial symmetric functions $m_\lambda(S)$.
Integer partitions with parts of size at most two will occur frequently; it
is convenient to use the notation $[n,i] = 2^{i} 1^{n-2i}$ for the partition of $n$
with $i$ parts of size $2$ and $n-2i$ parts of size $1$.

\begin{proof}[Proof of Theorem \ref{TWOFLATS}]
The basis enumerator of $\M$ is $M=\sum_{i,j} e_i(S) e_j(T)$, with the
sum over all pairs $(i,j)$ with $i+j=r$ and $0\leq i\leq s$ and $0\leq j\leq t$.
Since $s+t\geq r$ and $j=r-i$, the summation can be replaced with the sum over
$r-t\leq i\leq s$.  By Proposition \ref{GWS}, $M$ is stable if and only if the
bivariate $r$-form
$F(\alpha,\beta)=\sum_{i=r-t}^{s} \binom{a}{i} \binom{b}{r-i} \alpha^{i}\beta^{r-i}$
is stable.  Factor out $\beta^r$ and let $x=\alpha\beta^{-1}$ to see that
$F(\alpha,\beta)=\beta^r x^{r-t}P_\M(x)$.  As $\alpha$ and $\beta$ vary over all complex
numbers with positive imaginary part, $x$ varies over all complex numbers except
for nonpositive real numbers.  Thus, $F(\alpha,\beta)$ is stable if and only if
$P_\M(x)$ has only real nonpositive roots.  (Since the coefficients of $P_\M(x)$ are
positive, it has no positive real roots.)
\end{proof}
Corollary \ref{CORO2FLATS} follows immediately by applying the quadratic formula
to $P_\M(x)$.

The first claim of Theorem \ref{MAIN} follows from Corollary \ref{CORO2FLATS}, 
since $\M(r,\ell,a)$ is the case of $\M$ in Theorem \ref{TWOFLATS} in which $S=A$ and
$s=r$, and $T=L\cup\{e,f\}$ and $t=2$ and $b=\ell+2$.  Some routine calculation shows
that the condition in Corollary \ref{CORO2FLATS} holds if and only if either
$(r-2)\ell\leq 2$ or $a\leq\A(r,\ell)$.

By Proposition \ref{BWW}, a matroid is strongly Rayleigh if and only if every Rayleigh
difference is a PSD form.  In the rest of the paper, we begin to address the
question of when these Rayleigh differences are SOS forms. 

\begin{PROP}\label{UNIFORM}
For uniform matroids, every Rayleigh difference is a SOS form.
\begin{proof}
Let $\M=\U_{r,m}$ be the uniform matroid of rank $r$ on a set $E$ of size $m$;
its basis enumerator is $M = e_r(E)$.
By 2-transitivity of $\Sym_E$, only one Rayleigh difference $\Delta M\{e,f\}$
needs to be checked.  Fix $e,f\in E$, let $H=E\drop\{e,f\}$, let $e_\lambda=e_\lambda(H)$,
and let $d=r-1$.  Since $M_e^f=M_f^e=e_{d}$ and $M_{ef}=e_{d-1}$ and
$M^{ef}=e_{d+1}$, it follows that $\Delta M\{e,f\} = e_d^2 - e_{d-1}e_{d+1}$.
For $0\leq r\leq 2$ this is easily seen to be a sum of squares, so assume that
$d\geq 2$.  We claim that 
\begin{equation} \label{UNIF}
e_d^2 - e_{d-1} e_{d+1} =
\frac{1}{d+1} \sum_{j=0}^d \binom{d}{j}^{-1}\psi_{d,j}(H)
\end{equation}
in which $\psi_{d,j}(H)$ is defined for $0 \leq j \leq d$ by
$$
\psi_{d,j}(H) = \sum_{J \subseteq H: |J| = j} (\y^J)^2 e_{d-j}(H \drop J)^2.
$$
This expresses $\Delta M\{e,f\}$ as a SOS form.  Note that each $2d$-form $\psi_{d,j}(H)$
is a symmetric function of $\{y_h:\ h\in H\}$.

The first step is to express both sides of (\ref{UNIF}) in terms of the monomial basis
$\{ m_{\lb 2d,d-k \rb}:\ 0\leq k\leq d\}$.  One sees that 
$e_d^2 = \sum_{k=0}^d \binom{2k}{k} m_{\lb 2d,d-k \rb}$ and that
$e_{d-1}e_{d+1} = \sum_{k=1}^d \binom{2k}{k-1} m_{\lb 2d,d-k \rb}$, and it
follows that $e_d^2-e_{d-1}e_{d+1} = \sum_{k=0}^d \frac{1}{k+1}\binom{2k}{k} m_{\lb 2d,d-k \rb}$.
On the RHS of (\ref{UNIF}), the coefficient of $m_{\lb 2d,d-k \rb}$ in $\psi_{d,j}(H)$
is $\binom{d-k}{j}\binom{2k}{k}$.  To see this, let $U,V\subseteq H$ be disjoint
sets with $|U|=d-k$ and $|V|=2k$; the coefficient in question is the coefficient of
$(\y^U)^2\y^V$ in $\psi_{d,j}(H)$.  This monomial is constructed in
$\psi_{d,j}(H)$ as the product $(\y^J)^2 \y^S \y^T$ by choosing a $j$-subset
$J\subseteq U$ and a $k$-subset $K\subseteq V$, and setting $S=(U\drop J)\cup K$ and
$T=(U\drop J)\cup(V\drop K)$.  This construction is bijective:\ given such a
pair $(S,T)$ we recover $J=U\drop(S\cap T)$ and $K=S\drop T$.  There are
$\binom{d-k}{j}$ choices for $J$ and
$\binom{2k}{k}$ choices for $K$, establishing the formula.

Thus, equation (\ref{UNIF}) is equivalent to the statement that for all $0\leq k\leq d$,
$$
\frac{1}{d+1}\sum_{j=0}^d \binom{d}{j}^{-1}\binom{d-k}{j}\binom{2k}{k}
= \frac{1}{k+1}\binom{2k}{k}.
$$
Multiply both sides by $(d+1)\binom{d}{d-k}\binom{2k}{k}^{-1}$ and use
$\frac{d+1}{k+1}\binom{d}{d-k}=\binom{d+1}{k+1}$ and
$\binom{d}{d-k}\binom{d-k}{j}=\binom{d}{j}\binom{d-j}{k}$ to see that this is
equivalent to $\sum_{j=0}^d \binom{d-j}{k} = \binom{d+1}{k+1}$.
This is equivalent to $\sum_{j=0}^{d-k} \binom{d-j}{k} = \binom{d+1}{k+1}$,
since if $d-k<j\leq d$ then $\binom{d-j}{k}=0$.
This well-known binomial identity enumerates lattice paths from $(0,0)$ to
$(k+1,d-k)$, partitioned according to which of the edges
$(k,d-k-j)\goesto(k+1,d-k-j)$ is crossed, for each $0\leq j\leq d-k$.
\end{proof}
\end{PROP}

\subsection{Sums of squares}
Consider an arbitrary matroid $\M$ of rank $r=d+1$ on a set $E$ of size $m$, let
$\{e,f\}\subseteq E$, and let $H=E\drop\{e,f\}$ and $F(\y)=\Delta M\{e,f\}(\y)$.
This $F(\y)$ is a $2d$-form and each variable $\{y_h:\ h\in H\}$ occurs at most to
the second power.  Thus, one can write $F(\y)=\sum_\alpha F_\alpha\, \y^\alpha$
for some integers $F_\alpha\in\ZZ$ indexed by the functions
$\alpha:H\goesto\{0,1,2\}$ for which $|\alpha|=\sum_{h\in H} \alpha(h)=2d$,
and in which $\y^\alpha = \prod_{h\in H} y_h^{\alpha(h)}$.

Now assume that $F(\y) = \sum_{i=1}^n p_i(\y)^2$ is a SOS form.
It is not hard to see that each of the polynomials $p_i(\y)$ must be a multiaffine
$d$-form.  Thus, each $p_i(\y)$ for $1\leq i\leq n$ can be written
$$
p_i(\y) = \sum_{S \subseteq H:\ |S| = d} c_{(S,i)}\, \y^S
$$
for some real coefficients $c_{(S,i)} \in \RR$.  For each
$\alpha:H\goesto\{0,1,2\}$, let $\P(\alpha)$ be the set of pairs $(S,T)$ such
that $S,T\subseteq H$, $|S|=|T|=d$, $S\cap T=\alpha^{-1}(2)$, and $S\symdiff T =
\alpha^{-1}(1)$.  (Here $\symdiff$ denotes symmetric difference of sets.)  In
other words, $(S,T)\in\P(\alpha)$ if and only if $|S|=|T|=d$ and $\y^S\y^T =
\y^\alpha$.  Note that if $\y^\alpha$ is a monomial of $m_{[2d,d-k]}$ then
$|\P(\alpha)|=\binom{2k}{k}$.
It follows that for all $1\leq i\leq n$ and $\alpha:H\goesto\{0,1,2\}$,
$$
[\y^\alpha] p_i(\y)^2 = \sum_{(S,T)\in\P(\alpha)} c_{(S,i)}\, c_{(T,i)}.
$$
Consequently, for all $\alpha:H\goesto\{0,1,2\}$,
$$
F_\alpha = [\y^\alpha] F(\y) = \sum_{i=1}^n \sum_{(S,T)\in\P(\alpha)}
c_{(S,i)}\, c_{(T,i)}.
$$
For each $S \subseteq H$ with $|S| = d$, define the vector $\c_S \in \RR^n$ by
$(\c_S)_i = c_{(S,i)}$ for each $1\leq i\leq n$, and equip $\RR^n$ with its
usual Euclidean (dot) inner product.  The previous equation becomes
\begin{equation} \label{DOT-SYSTEM}
\sum_{(S,T)\in\P(\alpha)} \ev{\c_S,\c_T} = F_\alpha.
\end{equation}
Thus, the existence of a SOS expression for $F(\y)$ is equivalent to the existence
of a set of vectors $\{\c_S\in\RR^n:\ S\subseteq H\ \mathrm{and}\ |S|=d\}$
such that the inner products $\ev{\c_S,\c_T}$ satisfy a certain system
$\L$ of linear equations (\ref{DOT-SYSTEM}) for each $\alpha$,
with the RHSs of these equations determined by the coefficients of $F(\y)$.

For a $2d$-form $F(\y)=\Delta M\{e,f\}(\y)$ as above, the system
$\L$ has an unwieldy number of variables:\ $\binom{t+1}{2}$, in which $t=\binom{m-2}{d}$.
However, when $F(\y)$ has a large group of symmetries, as it does in our
case, there is an enriched system which is consistent if and only
if $\L$ is consistent, and which has significantly fewer free parameters.

\begin{LMA}\label{SYM}
Let $\v_1,...,\v_t$ be finitely many vectors in a Euclidean space $V$, and let
$\Gamma\leq\Sym_t$ be a group of permutations of $\{1,...,t\}$.  Let $\L$ be a
system of linear equations satisfied by the inner products $\ev{\v_i,\v_j}$ for
$1\leq i,j\leq t$ that is invariant under the action of $\Gamma$.  Then there
is a set of vectors $\w_1,...,\w_t$ in a Euclidean space $W$ such that\\
\textup{(i)}\  the inner products $\ev{\w_i,\w_j}$ satisfy $\L$, and\\ 
\textup{(ii)}\ for all $1\leq i,j\leq t$ and $\sigma\in\Gamma$,
$\ev{\w_{\sigma(i)},\w_{\sigma(j)}} = \ev{\w_i,\w_j}$.
\begin{proof}
For each $1\leq i,j\leq t$ let $\alpha_{ij}=\la \v_i,\v_j\ra$ and let
$A=(\alpha_{ij})$ be the Gram matrix of the vectors $\{\v_i\}$.
Let $\beta_{ij} = |\Gamma|^{-1}\sum_{\sigma\in\Gamma} \alpha_{\sigma(i),\sigma(j)}$,
and let $B=(\beta_{ij})$ be the corresponding matrix.
With $A_\sigma=(\alpha_{\sigma(i),\sigma(j)})$ for each $\sigma\in\Gamma$ we have
$B=|\Gamma|^{-1}\sum_{\sigma\in \Gamma} A_\sigma$, in which each matrix $A_\sigma$ is
positive semidefinite, and so $B$ is also positive semidefinite.
Therefore $B$ is the Gram matrix of some set of vectors $\{\w_i\}$ in some
Euclidean space $W$.

For any $\sigma\in\Gamma$, $\beta_{\sigma(i),\sigma(j)} = \beta_{ij}$,
and so the vectors $\{\w_i\}$ satisfy condition (ii).
Condition (i) follows from the $\Gamma$-invariance of $\L$:\
for any equation $\sum_{ij} c_{ij}\alpha_{ij} = \eta$ in $\L$ and
$\sigma\in\Gamma$, the equation $\sum_{ij} c_{ij}\alpha_{\sigma(i),\sigma(j)} = \eta$
is also in $\L$;\ it follows that
$$
\sum_{ij} c_{ij} \beta_{ij}
 =
\frac{1}{|\Gamma|}\sum_{\sigma\in \Gamma} \sum_{ij} c_{ij}\alpha_{\sigma(i),\sigma(j)}
 =
\frac{1}{|\Gamma|}\sum_{\sigma\in \Gamma} \eta = \eta,
$$
as required.
\end{proof}
\end{LMA}

\subsection{Uniform matroids and Johnson schemes}
We return to the case of uniform matroids, for later use.
Let $\M=\U_{r,m}$ be the uniform matroid of rank $r=d+1$ on a set $E$ of
size $m=v+2$.  Adopting the notation above,
$\Delta M\{e,f\}=e_d^2-e_{d-1}e_{d+1}$ is invariant under the symmetric group
$\Sym_H$.  The orbits of the induced action of $\Sym_H$ on pairs $(S,T)$ of
$d$-subsets of $H$ are indexed by the integers $0\leq k\leq s =\min\{d,v-d\}$,
with the orbit indexed by $k$ corresponding to those pairs $(S,T)$ with $|S\cap T|=d-k$.
For each $0\leq k\leq s$, let $A_k$ be the square matrix indexed by $d$-subsets
of $H$, and with $(S,T)$-entry
$$
(A_k)_{S,T} = \left\{
\begin{array}{ll}
1 & \mathrm{if}\ |S\cap T| = d-k,\\
0 & \mathrm{otherwise}.
\end{array}
\right.
$$
These are the adjacency matrices of the Johnson association scheme $J(v,d)$;\
see \cite{Go,Z}.
They are symmetric and pairwise commuting, and hence simultaneously diagonalizable;\
$A_0=I$ is the identity matrix and $\sum_{k=0}^s A_k = J$ is the all-ones
matrix.  For each $0\leq k\leq s$, the all-ones vector $\one$ is an
eigenvector of $A_k$ with eigenvalue $\binom{d}{d-k}\binom{v-d}{k}$.

By Proposition \ref{UNIFORM}, since $\M=\U_{r,m}$, $\Delta M\{e,f\}$ is a SOS form.
Let $\{\c_S\in\RR^n:\ S\subseteq H\ \mathrm{and}\ |S|=d\}$ be the 
corresponding set of vectors as in Section 2.3.
Consider any $\alpha:H\goesto\{0,1,2\}$ with
$|\alpha|=2d$, and let $|\alpha^{-1}(2)|=d-k$ for some $0\leq k\leq s$.
Since $e_d^2-e_{d-1}e_{d+1} = \sum_{k=0}^d \frac{1}{k+1}\binom{2k}{k} m_{[2d,d-k]}$,
it follows that
\begin{equation} \label{UNIF-JOHN}
\sum_{(S,T)\in\P(\alpha)} \ev{\c_S,\c_T} = [\y^\alpha]\Delta M\{e,f\} =\frac{1}{k+1}\binom{2k}{k}.
\end{equation}
This system of linear equations is invariant under the action of $\Sym_H$, and
so by Lemma \ref{SYM} there is a solution $\{\c_S\}$ such that $\ev{\c_S,\c_T}$
depends only on $|S\cap T|=d-k$.  There are $\binom{2k}{k}$ terms on the LHS
of (\ref{UNIF-JOHN}), all in the same orbit of $\Sym_H$, and it follows
that in this symmetrized solution $\ev{\c_S,\c_T} = 1/(k+1)$ whenever $|S\cap
T|=d-k$.  In other words, in terms of the matrices of the Johnson scheme,
the Gram matrix of $\{\c_S\}$ is
$$
G = A_0 + \frac{1}{2}A_1 + \cdots + \frac{1}{s+1} A_s.
$$

\begin{PROP} \label{JOHN}
Let $0\leq d\leq v$, let $s=\min\{d,v-d\}$, and let $\{A_0,A_1,...,A_s\}$ be the
adjacency matrices of the Johnson scheme $J(v,d)$.  Then the matrix
$G=\sum_{k=0}^s \frac{1}{k+1}A_k$ is positive semidefinite and $\one$ is an
eigenvector for $G$ with eigenvalue $\frac{1}{d+1}\binom{v+1}{d}$.
\end{PROP}
\begin{proof}
The preceding remarks of this section show that $G$ is positive semidefinite and
that $\one$ is an eigenvector for $G$ with eigenvalue
$\sum_{k=0}^s \frac{1}{k+1}\binom{d}{d-k}\binom{v-d}{k}$.
If $v-d=s<k\leq d$ then $\binom{v-d}{k}=0$, so this summation can be extended to
$0\leq k\leq d$ in either case.
To complete the proof it suffices to show that
$$
\sum_{k=0}^d \frac{1}{k+1} \binom{d}{d-k}\binom{v-d}{k} =
\frac{1}{d+1}\binom{v+1}{d}.
$$
Multiplying both sides by $d+1$ and using
$\frac{d+1}{k+1}\binom{d}{d-k}=\binom{d+1}{d-k}$, it suffices to show that
$\sum_{k=0}^d \binom{d+1}{d-k}\binom{v-d}{k} = \binom{v+1}{d}$.
This well-known binomial identity enumerates lattice paths from $(0,0)$ to
$(d,v+1-d)$, with the $k$-th summand enumerating those paths which pass through
the point $(d-k,k+1)$, for each $0\leq k\leq d$.  Each lattice path from $(0,0)$
to $(d,v+1-d)$ passes through exactly one of these points, proving the result.
\end{proof}

\section{Proof of Theorem 1.}

\subsection{Analyzing the SOS equations} \label{ANALYZE-SOS}

Let $\M=\M(r,\ell,a)$ be the matroid in Theorem~\ref{MAIN}, with $L$ and $A$ as
in the statement and with $d=r-1$ and $H=E\drop\{e,f\}= L\cup A$.  The basis enumerator of $\M$ is 
\begin{equation} \label{BASIS-ENUM}
M = e_{r}(A) + e_1(L\cup\{e,f\})e_{r-1}(A) + e_2(L\cup\{e,f\})e_{r-2}(A).
\end{equation}

\begin{LMA} \label{DELTA-M}
The Rayleigh difference of $\{e,f\}$ in $\M$ is
\begin{eqnarray*}
\Delta M\{e,f\}
&=& \sum_{k=0}^d \frac{1}{k+1}\binom{2k}{k} m_{\lb 2d, d-k \rb}(A) \\
& & +\, m_1(L)\, \sum_{k=1}^{d} \binom{2k-1}{k} m_{\lb 2d-1, d-k \rb}(A) \\
& & +\, (m_2(L) + m_{11}(L))\, \sum_{k=0}^{d-1} \binom{2k}{k} m_{\lb 2d-2, d-1-k \rb}(A).
\end{eqnarray*}
\end{LMA}
\begin{proof}
From (\ref{BASIS-ENUM}), we see that 
\begin{eqnarray*}
M_e^f = M_f^e &=& e_{r-1}(A) + e_1(L)e_{r-2}(A), \\
M_{ef} &=& e_{r-2}(A), \\
\mathrm{and}\ \ M^{ef} &=& e_r(A) + e_1(L)e_{r-1}(A) + e_2(L)e_{r-2}(A).
\end{eqnarray*}
It follows that $M_e^f M_f^e -M_{ef}M^{ef}$ equals
$$
e_d(A)^2 - e_{d-1}(A)e_{d+1}(A)
+ e_1(L)e_d(A)e_{d-1}(A)
+ (e_1(L)^2 - e_2(L)) e_{d-1}(A)^2.
$$
\noindent
Arguments analogous to those in the proof of Proposition \ref{UNIFORM}
finish the calculation.
\end{proof}

Now $\Delta M \{e,f\} = \sum_{i=1}^n p_i(\y)^2$ is a SOS form if and only if
there is a corresponding set of vectors $\{\c_S\in\RR^n:\ S\subseteq H\
\mathrm{and}\ |S|=d\}$ as in Section 2.3.
Let $\L$ denote the system of linear equations (\ref{DOT-SYSTEM})
induced by comparison of coefficients, with RHSs given by $F_\alpha =
[\y^\alpha]\Delta M\{e,f\}$.

Consider any $\alpha:H\goesto\{0,1,2\}$ with
$|\alpha|=2d$, and let $U=\alpha^{-1}(2)$ and $V=\alpha^{-1}(1)$.  Let
$|U|=d-k$, so that $0\leq k\leq d$ and $|V|=2k$.
From Lemma \ref{DELTA-M}, the equations
$\sum_{(X,Y)\in\P(\alpha)}\ev{\c_X,\c_Y} = F_\alpha$ of $\L$
fall into several cases, as follows.
\begin{eqnarray}
\label{C1} & & \mbox{If $U\subseteq A$ and $V\subseteq A$, then $F_\alpha=\frac{1}{k+1}\binom{2k}{k}$.}\\
\label{C2} & & \mbox{If $U\subseteq A$ and $|V\cap L|=1$, then $F_\alpha = \binom{2k-1}{k}$.}\\
\label{C3} & & \mbox{If $U\subseteq A$ and $|V\cap L|=2$, then $F_\alpha = \binom{2k-2}{k-1}$.}\\
\label{C4} & & \mbox{If $|U\cap L|=1$ and $V\subseteq A$,  then $F_\alpha = \binom{2k}{k}$.}\\
\label{C5} & & \mbox{In all remaining cases, $F_\alpha=0$.}
\end{eqnarray}

One sees that $\Delta M\{e,f\}$ is invariant under the Young subgroup
$\Gamma=\Sym_L\times\Sym_A$ of $\Sym_H$.
If $(U,V)$ is a pair as in cases (\ref{C1}) to (\ref{C5}) and $\sigma\in\Gamma$,
then $(\sigma(U),\sigma(V))$ is another such pair, and is in the same one of these
cases as is $(U,V)$.  It follows that the system $\L$ of linear equations is invariant
under the action of $\Gamma$.  By Lemma \ref{SYM} we may enrich $\L$ by the requirement
that $\ev{\c_S,\c_T}$ depends only on the orbit of $(S,T)$ in the action of
$\Gamma$ on pairs of $d$-subsets of $H$ without introducing a new inconsistency.
Since $\ev{\c_S,\c_T}=\ev{\c_T,\c_S}$, this common value on the orbit of $(S,T)$
is the same as the common value on the orbit of $(T,S)$.
For a pair $(S,T)$ of $d$-subsets of $H$, let
$$
\omega(S,T)=(|S\cap L|,|T\cap L|,|S\cap T\cap L|, |S\cap T\cap A|).
$$
Two such pairs $(S,T)$ and $(X,Y)$ are in the same orbit of $\Gamma$ if and only
if $\omega(S,T)=\omega(X,Y)$.

\subsection{The putative Gram matrix} \label{GRAM-SECT}
We continue with the notation of Section \ref{ANALYZE-SOS}.  Also, the notation
$\S(H,d) = \{S\subseteq H:\ |S|=d\}$ will be convenient.

\begin{LMA} \label{SYM-SYSTEM}
Let $\{\c_S\in\RR^n:\ S\in\S(H,d)\}$ be a set of vectors solving the equations
$\L$ of (\ref{C1}) to (\ref{C5}).  Let $S,T\in\S(H,d)$ and let $p\in L$.\\
\textup{(a)}  Then $\ev{\c_S,\c_S}=1$ if $|S\cap L|\leq 1$, and
$\c_S=\zero$ if $|S\cap L|\geq 2$.\\
\textup{(b)}  If $S\cap L = T\cap L = \{p\}$, then $\c_S=\c_T$.
\end{LMA}
\begin{proof}
For (a),
the inner product $\ev{\c_S,\c_S}$ corresponds to
$\y^\alpha = (\y^S)^2$ and $\P(\alpha)=\{(S,S)\}$.
This corresponds to $U=S$ and $V=\none$ and $k=0$.  If $S\cap L=\none$ then case
(\ref{C1}) applies and $\ev{\c_S,\c_S}=\frac{1}{0+1}\binom{0}{0}=1$.
If $|S\cap L|=1$ then case (\ref{C4}) applies and
$\ev{\c_S,\c_S}=\binom{0}{0}=1$.  If $|S\cap L|\geq 2$ then case (\ref{C5}) applies
and $\ev{\c_S,\c_S}=0$, so that $\c_S=\zero$.

For (b), let $\y^\alpha=\y^S\y^T$, and define $U$, $V$, and $k$ accordingly from
$\alpha$; this is in case (\ref{C4}) above.
In the equation $\sum_{(X,Y)\in\P(\alpha)} \ev{\c_X,\c_Y}
= \binom{2k}{k}$ the LHS has $\binom{2k}{k}$ terms, all in the same orbit of
$\Gamma$ as $(S,T)$.  It follows that $\ev{\c_S,\c_T}=1$.  Since
$\ev{\c_S,\c_S}=\ev{\c_T,\c_T}=1$, it follows that $\c_S=\c_T$.
\end{proof}
For each $p\in L$, denote by $\c_p\in\RR^n$ the vector such that $\c_S=\c_p$ for
all $S\in\S(H,d)$ for which $S\cap L = \{p\}$.

\begin{PROP} \label{GRAM}
With the notation above, $\Delta M\{e,f\}$ is a SOS form if and only if
there are unit vectors $\{\c_S\in\RR^n:\ S\in\S(A,d)\}$
and $\{\c_p\in\RR^n:\ p\in L\}$ such that the following hold.\\
\textup{(a)} For $S,T\in\S(A,d)$ with $|S\cap T|=d-k$, $\ev{\c_S,\c_T} = 1/(k+1)$.\\
\textup{(b)} For $S\in\S(A,d)$ and $p\in L$, $\ev{\c_S,\c_p}=1/2$.\\
\textup{(c)} For $p,q\in L$ with $p\neq q$, $\ev{\c_p,\c_q} = 1/2$.
\end{PROP}
\begin{proof}
We have seen that $\Delta M\{e,f\}$ is a SOS form if and only if $\L$ has a
solution that is constant on orbits of $\Gamma$ acting on pairs of $d$-subsets
of $H$.  By Lemma \ref{SYM-SYSTEM} such a solution must consist of unit vectors
$\{\c_S\}$ and $\{\c_p\}$ indexed as in the statement.  The remaining equations
from $\L$ are equivalent to (a), (b), and (c), as follows.

For (a), let $\y^\alpha=\y^S\y^T$, and define $U$, $V$, and $k$ accordingly from
$\alpha$; this is in case (\ref{C1}) above.  In the equation $\sum_{(X,Y)\in\P(\alpha)} \ev{\c_X,\c_Y}
= \frac{1}{k+1}\binom{2k}{k}$ the LHS has $\binom{2k}{k}$ terms, all in the same orbit of
$\Gamma$ as $(S,T)$.  It follows that $\ev{\c_S,\c_T}=1/(k+1)$.

For (b), let  $T\subseteq H$ be such that $T\cap L=\{p\}$.
Let $\y^\alpha=\y^S\y^T$, and define $U$, $V$, and $k$ accordingly from
$\alpha$;  this is in case (\ref{C2}) above.  In the equation
$\sum_{(X,Y)\in\P(\alpha)} \ev{\c_X,\c_Y} = \binom{2k-1}{k}$ the LHS has $\binom{2k}{k}$ terms,
all in the same orbit of $\Gamma$ as either $(S,T)$ or $(T,S)$.  It follows that
$\ev{\c_S,\c_T}=\binom{2k-1}{k}/\binom{2k}{k}=1/2$.  This is independent of the
choice of $T$, so $\ev{\c_S,\c_p}=1/2$ is self-consistent.

For (c), let $S,T\subseteq H$ be such that $S\cap L=\{p\}$ and $T\cap L=\{q\}$.
Let $\y^\alpha=\y^S\y^T$, and define $U$, $V$, and $k$ accordingly from
$\alpha$; this is in case (\ref{C3}) above.
In the equation $\sum_{(X,Y)\in\P(\alpha)} \ev{\c_X,\c_Y} = \binom{2k-2}{k-1}$
the LHS has $\binom{2k}{k}$ terms, but if $\{p,q\}\subseteq X$ or
$\{p,q\}\subseteq Y$ then $\c_X=\zero$ or $\c_Y=\zero$, so that
$\ev{\c_X,\c_Y}=0$.  There are $2\binom{2k-2}{k-1}$ other terms, obtained by
choosing a $(k-1)$-subset $X'\subseteq V\cap A$, letting $Y'= (V\cap A)\drop
X'$, and considering the pairs $(X'\cup\{p\},Y'\cup\{q\})$ and $(X'\cup\{q\},Y'\cup\{p\})$.
Each of these pairs is in the same orbit of $\Gamma$ as $(S,T)$, and
it follows that $\ev{\c_S,\c_T}=\binom{2k-2}{k-1}/2\binom{2k-2}{k-1}=1/2$.
This is independent of the choice of $S$ and $T$, so $\ev{\c_p,\c_q}=1/2$ is
self-consistent.

This shows that (a), (b), and (c) are necessary.
Conversely, assume that $\{\c_S:\ S\in\S(A,d)\}$ and $\{\c_p:\ p\in L\}$ are
unit vectors as in the statement of the proposition, and for
$S\in\S(H,d)\drop\S(A,d)$ let $\c_S=\c_{p}$ if $S\cap L=\{p\}$ and let
$\c_S=\zero$ if $|S\cap L|\geq 2$.  This set $\{\c_S:\ S\in\S(H,d)\}$ is a
solution to $\L$, as is easily checked.  As in the previous three paragraphs,
cases (\ref{C1}), (\ref{C2}), and (\ref{C3}) follow from (a), (b), and (c).
Case (\ref{C4}) follows from $\c_S=\c_T$ whenever $S\cap L=T\cap L=\{p\}$ as in
the proof of Lemma \ref{SYM-SYSTEM}, and case (\ref{C5}) follows from
$\c_S=\zero$ when $|S\cap L|\geq 2$.

This completes the proof.
\end{proof}

Imagine the vectors $\{\c_S\in\RR^n:\ S\in\S(A,d)\}$ in any order, followed by
the vectors $\{\c_p:\ p\in L\}$ in any order, and form the putative Gram matrix
$\G$ of their inner products.  By Proposition \ref{GRAM}, this matrix has the block form
\begin{equation} \label{GRAM-MAT}
\G = \frac{1}{2}
\left[ \begin{array}{cc}
2G_t & J_{t\times\ell} \\
J_{\ell\times t} & I_\ell + J_\ell
\end{array} \right]
\end{equation}
in which $t=\binom{a}{d}$.  The upper-left block is $2$ times the $t$-by-$t$ square
matrix $G_t$ of Proposition \ref{JOHN} (with $v=a$) for the Johnson scheme $J(a,d)$.
The lower-right block is $\ell$-by-$\ell$ square, $I_\ell$ is the identity
matrix, and the various $J$ matrices are all-ones matrices of the appropriate
shapes.  We have reduced the problem to the following.

\begin{CORO} \label{GRAM-CORO}
With the notation above, $\Delta M\{e,f\}$ is a SOS form if and only if the
matrix $\G$ of (\ref{GRAM-MAT}) is positive semidefinite.
\end{CORO}

\subsection{D\'enouement}
The matrix $I_\ell + J_\ell$ has eigenvalues $\ell+1$ of multiplicity one and
$1$ of multiplicity $\ell-1$, and thus is positive definite and hence
invertible. The matrix $\G$ is thus positive semidefinite if and only if the Schur complement
$$
C_t = G_t - \frac{1}{2} J_{t\times\ell}(I_\ell + J_\ell)^{-1}J_{\ell\times t}
$$ 
is positive semidefinite (see item (0.8.5) of \cite{HJ}).  One easily checks
that $(I_\ell+J_\ell)^{-1} = I_\ell - (\ell+1)^{-1}J_\ell$, and that
the Schur complement in question is
\begin{equation} \label{SCHUR-C}
C_t = G_t -\frac{\ell}{2\ell+2}J_t.
\end{equation}

\begin{LMA} \label{PSD-INEQ}
With the notation above, $\Delta M\{e,f\}$ is a SOS form if and only if
$\frac{1}{d+1}\binom{a+1}{d}\geq \frac{\ell}{2\ell+2}\binom{a}{d}$.
\end{LMA}
\begin{proof}
By Corollary \ref{GRAM-CORO}, it suffices to determine when the matrix $\G$ of
(\ref{GRAM-MAT}) is positive semidefinite.  By the remarks of this section,
it suffices to determine when the matrix $C_t$ in (\ref{SCHUR-C}) is positive
semidefinite.  Let $\{\u_1,...,\u_t\}$ be an orthogonal basis of $\RR^t$
consisting of eigenvectors for $G_t$, with $\u_1=\one$, and let
$\theta_1=\frac{1}{d+1}\binom{a+1}{d},\theta_2,...,\theta_t$ be the
corresponding eigenvalues;\ by Proposition \ref{JOHN}, $\theta_i\geq 0$ for all
$1\leq i\leq t$.  Since this basis is orthogonal, the $\{\u_i\}$ are
eigenvectors of $J_t$ as well, with corresponding eigenvalues
$\xi_1=\binom{a}{d}$ and $\xi_i=0$ for all $2\leq i\leq t$.  Thus, the
$\{\u_i\}$ are an orthogonal basis of $\RR^t$ consisting of eigenvectors of
$C_t$, with corresponding eigenvalues $\theta_i-\ell\xi_i/(2\ell+2)$ for
$1\leq i\leq t$.  These eigenvalues are all nonnegative if and only if
$\theta_1\geq \ell\xi_1/(2\ell+2)$.  This proves the result.
\end{proof}

\begin{proof}[Proof of Theorem \ref{MAIN}]
Using $\binom{a+1}{d}=\frac{a+1}{a+1-d}\binom{a}{d}$ and $a+1-d\geq 0$, elementary
calculation shows that the inequality of Lemma \ref{PSD-INEQ} is equivalent to
$$
a((r-2)\ell-2) \leq r((r-2)\ell-2)+2(r+\ell+1).
$$
When $(r,\ell)$ is one of $(3,1)$, $(3,2)$, or $(4,1)$ the factor
$(r-2)\ell-2\leq 0$ is nonpositive and the inequality holds.  In all other
cases $(r-2)\ell-2>0$ is positive, and the inequality is equivalent to
$a\leq \A(r,\ell)$, as claimed.  By Lemma \ref{PSD-INEQ}, this establishes the
first claim of the theorem.

We prove that $\M(r,\ell,a)$ is strongly Rayleigh when either
$(r-2)\ell\leq 2$ or $a\leq \A(r,\ell)$ by induction on $\ell$.
For the basis of induction it is convenient to take the degenerate
case $\ell=0$;\ then $\M(r,0,a)=\U_{r,a+2}$ is a uniform matroid, which is strongly
Rayleigh by Proposition \ref{UNIFORM}.  For the induction step, assume that
$\ell\geq 1$, and that if $\M(r,\ell-1,a)$ satisfies either $\ell-1=0$,
$(r-2)(\ell-1)\leq 2$, or $a\leq \A(r,\ell-1)$, then
$\M(r,\ell-1,a)$ is strongly Rayleigh.  Let $\M=\M(r,\ell,a)$
be such that either $(r-2)\ell\leq 2$ or $a\leq \A(r,\ell)$.

By the first claim of the theorem, $\Delta M\{e,f\}$ is a SOS form, hence PSD.
By Proposition \ref{BWW}(d), to complete the proof it suffices to show that $M_e$,
$M^e$, $M_f$, and $M^f$ are strongly Rayleigh.  By symmetry, it suffices to
show that $M_e$ (the basis enumerator of the contraction $\M/e$)
and $M^e$ (the basis enumerator of the deletion $\M\drop e$) are strongly Rayleigh.

The contraction $\M/e$ is the uniform matroid $\U_{r-1,a+1}$ with the point
corresponding to the image of $f$ fattened to a parallel class of size
$\ell+1$.  The simplification of $\M/e$ is thus $\U_{r-1,a+1}$, which is strongly
Rayleigh, with basis enumerator $e_{r-1}(A\cup\{f\})$.  The basis enumerator
$M_e$ of $\M/e$ is obtained from $e_{r-1}(A\cup\{f\})$ by substituting
$y_f=y_1+\cdots+y_{\ell+1}$ for new variables $\{y_1,...,y_{\ell+1}\}$.
One can check that this operation preserves the strong
Rayleigh property.  (In general, a matroid is strongly Rayleigh if and only if
its simplification is strongly Rayleigh.)  Thus, $M_e$ is strongly Rayleigh.

For the deletion $\M\drop e$, a short calculation shows that
$$
\frac{\partial}{\partial\ell}\A(r,\ell) = \frac{-2r(r-1)}{((r-2)\ell-2)^2}<0
$$
since $r\geq 3$.  Therefore, $\M\drop e = \M(r,\ell-1,a)$ satisfies either
$\ell-1=0$,  $(r-2)(\ell-1)\leq 2$, or $a\leq\A(r,\ell)<\A(r,\ell-1)$.
In any case the induction hypothesis applies, so that $M^e$ is strongly Rayleigh.

This completes the induction step, and the proof.
\end{proof}

\section{Concluding Remarks.}

For any matroid, stability of the basis enumerator is a complex analytic criterion 
for all the Rayleigh differences to be PSD forms, while the method of Section
2.3 is a geometric criterion for some Rayleigh difference to be a SOS form.
It is a fascinating interaction.

The strategy of our proof can naturally be extended to more complicated cases.
Among these, the following simple matroids are perhaps tractable.

(i)\ $\M$ consists of $\ell+2\geq t+1$ points in general position on a flat of rank $t$,
extended by $a\geq r-t$ points in general position (relative to the flat) in rank $r$.
(The case we consider is $t=2$.)  The Young subgroup has two orbits on points,
but the putative Gram matrix as in Section \ref{GRAM-SECT}
has a more complicated block structure.  Determining whether the basis
enumerator is stable, as in Theorem \ref{TWOFLATS}, involves determining whether or not a
particular univariate polynomial of degree $t$ has only real roots.

(ii)\ Represented over the reals, $\M$ consists of $a\geq r$ points in a subspace $U$
of dimension $r-1$, and  $b\geq s$ points in a subspace $V$ of dimension $s-1$, in
as general position as possible subject to $\dim(U\cap V)=c$.  (When $c\in\{0,1\}$ this is
a direct sum or $2$-sum of uniform matroids, and hence is strongly Rayleigh.
When $c=s-1=t$ this reduces to (i) above.)

(iii)\ $\M$ consists of $c\geq r$ copunctal lines in rank $r\geq 3$, with lines
and points in as general position as possible.  In the case $c=r=3$ this is known
to be strongly Rayleigh -- see Corollary 10.3 and Example 10.4 of \cite{COSW}.

\end{document}